\documentclass{icmart}
\usepackage{epsfig,url}

\contact[brown@ihes.fr]{IHES, 35 Route de Chartres, Bures-sur-Yvette, 91440 France}

\numberwithin{equation}{section}

\newtheorem{theorem}{Theorem}[section]
\newtheorem{defn}{Definition}[theorem]
\newtheorem{corollary}[theorem]{Corollary}
\newtheorem{lemma}[theorem]{Lemma}
\newtheorem{example}[theorem]{Example}
\newtheorem{proposition}[theorem]{Proposition}
\newtheorem{conjecture}[theorem]{Conjecture}

\newtheorem*{coro}{Corollary} 

\theoremstyle{definition}
\newtheorem{definition}[theorem]{Definition}
\newtheorem{remark}[theorem]{Remark}
\font \rus= wncyr10

\newcommand{\Pro}{\mathbb{P}}
\newcommand{\ZZ}{\mathbb{Z}}
\newcommand{\QQ}{\mathbb{Q}}
\newcommand{\GG}{\mathbb{G}}
\newcommand{\CC}{\mathbb{C}}
\newcommand{\Au}{\mathcal{A}}
\newcommand{\RR}{\mathbb{R}}
\newcommand{\HH}{\mathbb{H}}
\newcommand{\Ho}{\mathcal{H}}
\newcommand{\NN}{\mathbb{N}}
\newcommand{\MT}{\mathcal{MT}}
\newcommand{\aaa}{\mathfrak{a}}
\newcommand{\tone}{\overset{\rightarrow}{1}\!}
\newcommand{\opi}{{}_0 \Pi_{1}}
\newcommand{\mm}{\mathfrak{m}}
\newcommand{\zetam}{\zeta^{\mm}}
\newcommand{\uu}{\mathfrak{u}}
\newcommand{\Or}{\mathcal{O}}
\newcommand{\dch}{\mathrm{dch}}
\newcommand{\Lef}{\mathbb{L}}
\newcommand{\per}{\mathrm{per}}
\newcommand{\Pe}{\mathcal{P}}
\newcommand{\dd}{\mathcal{D}}
\newcommand{\gr}{\mathrm{gr}}
\newcommand{\gm}{\mathfrak{g}}
\newcommand{\Isom}{\mathrm{\underline{Isom}}}
\newcommand{\sha}{\, \hbox{\rus x} \,}

\newcommand{\bq}{\backslash \!\!\backslash }

\title[ Motivic periods and  $\Pro^1\backslash \{0,1,\infty\}$ ]{Motivic periods and  $\Pro^1\backslash \{0,1,\infty\}$}

\author[Francis Brown]
{Francis Brown\thanks{Beneficiary of ERC Grant 257638}}

\begin{document}

\begin{abstract}
This is a  review of the theory of the motivic fundamental group of the projective line minus three points, 
and its relation to multiple zeta values.
\end{abstract}



\maketitle

\section{Introduction}
The role of the projective line minus three points $X=\Pro^1 \backslash \{0,1,\infty\}$ in relation to Galois theory can be traced back to Belyi's theorem  \cite{Belyi} (1979):
\begin{theorem}  \label{thmBelyi}
Every smooth projective algebraic curve defined over $\overline{\QQ}$ can be  realised as a ramified cover of $\Pro^1$, whose ramification  locus is contained in  $\{0,1,\infty\}$.
  \end{theorem} 
Belyi deduced that the absolute Galois group of $\QQ$ acts faithfully on the profinite completion
of the fundamental group of $X$, i.e., the map
\begin{equation} \label{introGQaction} 
\mathrm{Gal}(\overline{\QQ} /\QQ) \rightarrow  \mathrm{Aut}(\widehat{\pi}_1(X(\CC), b))
\end{equation} 
where $b \in X(\QQ)$, is injective. In his famous proposal `\emph{Esquisse d'un programme}'  in 1984 \cite{Esquisse}, Grothendieck  suggested 
studying the absolute Galois group of $\QQ$ via its action on completions of  fundamental groups of moduli spaces of curves $\mathcal{M}_{g,n}$
of genus $g$ with $n$ ordered marked points  ($X$ being isomorphic to $\mathcal{M}_{0,4}$) and their interrelations.   A few years later, at approximately the same time,  
these ideas were developed in somewhat different directions  in three enormously influential papers due to Drinfeld, Ihara, and Deligne \cite{Drinfeld, Ihara, DeP1}.
Ihara's 1990 ICM  talk gives a  detailed  account of the subject  at that time \cite{IharaICM}.  However, the problem of determining the image of the map $(\ref{introGQaction})$   remains completely open to this day.

In this talk I will mainly consider the pro-unipotent  completion of the fundamental group of $X$, which seems to be  a more tractable object 
than its profinite version, and closely follow the point of view of Deligne, and Ihara (see \cite{IharaICM}, \S5).

\subsection{Unipotent completion} \label{sectIntroUnip}
Deligne showed \cite{DeP1}  that the pro-unipotent completion of $\pi_1(X)$  carries many extra structures corresponding to the realisations of an (at the time)  hypothetical 
category  of mixed Tate motives over the integers. Since then,  the motivic framework has now been completely established due to the work of a large number of different authors including Beilinson, Bloch, Borel, Levine, Hanamura,  and Voevodsky. The definitive reference is  \cite{DG}, \S\S1-2.

\begin{enumerate}
\item There exists an abstract  Tannakian category $\MT(\ZZ)$ of mixed Tate motives unramified  over $\ZZ$. It is a $\QQ$-linear  
subcategory of the category $\MT(\QQ)$ of mixed Tate motives over $\QQ$ obtained by restricting certain Ext groups.
It is equivalent to the category of representations of  an affine group scheme $G^{dR}$ which is defined over $\QQ$ and is a semi-direct product 
$$G^{dR} \cong U^{dR} \rtimes \GG_m\ .$$
The subgroup $U^{dR}$ is pro-unipotent, and its graded Lie algebra (for the action of  $\GG_m$) is isomorphic to the free graded Lie algebra with one generator
$$\sigma_3, \sigma_5, \sigma_7, \ldots $$
in every odd negative degree $\leq -3$. 
The essential reason for this is that the algebraic $K$-theory of the integers   $K_{2n-1}(\ZZ)$ has rank
1 for  $n =3,5,7,\ldots $, and  rank 0  otherwise, as shown by Borel \cite{Bo1,Bo2}.
Note that the elements $\sigma_{2n+1}$ are only well-defined modulo commutators.

\item  The pro-unipotent completion $\pi^{un}_1(X,\tone_0, -\tone_1)$ is the Betti realisation of an object, called the motivic fundamental groupoid (denoted by $\pi_1^{mot}$), 
whose affine ring is a limit of  objects in the category $\MT(\ZZ)$. 
\end{enumerate}

The majority  of these notes will go into explaining    2 and  some of  the ideas behind the following motivic analogue of Belyi's injectivity theorem   $(\ref{introGQaction})$:

\begin{theorem} \label{thmfaithfulactionofGdr} $G^{dR}$ acts faithfully on the de Rham realisation of  $\pi^{mot}_1(X, \tone_0, -\tone_1)$.
\end{theorem} 

This theorem has an   $\ell$-adic version which can be translated into  classical Galois theory (\cite{IharaICM}, \S5.2), and relates to some
questions  in the literature  cited above. Unlike Belyi's theorem, which is  geometric, the  proof of  theorem \ref{thmfaithfulactionofGdr} is  arithmetic and  combinatorial.
The main ideas came from the theory of multiple zeta values.

\subsection{Multiple zeta values}
Let $n_1,\ldots, n_r$ be integers $\geq 1$ such that $n_r\geq 2$. Multiple zeta values are defined by the convergent nested sums
$$\zeta(n_1,\ldots, n_r) = \sum_{1\leq k_1 < \ldots < k_r} {1 \over k_1^{n_1} \ldots k_r^{n_r}}  \quad \in \RR \  .$$
The  quantity $N=n_1+\ldots+ n_r$ is known as the weight, and $r$ the depth.
Multiple zeta values were first studied by Euler (at least in the case $r=2$)  and were rediscovered independently in mathematics by Zagier and  Ecalle, and 
 in perturbative quantum field theory by Broadhurst and Kreimer.
 They satisfy a vast array of algebraic relations which are not completely understood at the time of writing.

The relationship between these numbers and  the fundamental group comes via the theory of iterated integrals, which are implicit in the work  of Picard and were rediscovered  by Chen and Dyson.
In general, let $M$ be a differentiable manifold and let $\omega_1, \ldots, \omega_n$ be smooth 1-forms on $M$.
Consider a smooth path $\gamma: (0,1) \rightarrow M$. The iterated integral of $\omega_1,\ldots, \omega_n$ along $\gamma$
 is defined (when it converges) by
$$ \int_{\gamma} \omega_1 \ldots \omega_n = \int_{0< t_1 < \ldots < t_n < 1} \gamma^*(\omega_1)(t_1) \ldots  \gamma^*(\omega_n)(t_n)\ .$$
 Kontsevich observed that when $M=X(\CC)$ and  $\gamma(t)=t$ is simply the inclusion $(0,1) \subset X(\RR)$,  one has the following integral representation
\begin{equation} \label{MZVasitint}
\zeta(n_1,\ldots, n_r) = \int_{\gamma} \omega_1 \underbrace{\omega_0 \ldots \omega_0}_{n_1-1}  \omega_1  \underbrace{\omega_0 \ldots \omega_0}_{n_2-1}   \ldots\omega_1  \underbrace{\omega_0 \ldots \omega_0}_{n_r-1} 
\end{equation}
where
$\omega_0 = {dt \over t}$ and $\omega_1 = {dt \over 1-t}$.
I will explain in \S\ref{Periodsofpi1} that this formula allows one to interpret multiple zeta values as periods of the pro-unipotent fundamental groupoid of $X$. 
The  action of the motivic Galois group $G^{dR}$ on the (de Rham version of) the latter should translate, via Grothendieck's period conjecture, into an  
action on multiple zeta values themselves. Thus one expects multiple zeta values to be a basic example in  a Galois theory of transcendental numbers (\cite{An}, \S23.5); the action of the Galois group should preserve all their algebraic relations.

 Of course,  Grothendieck's period conjecture  is not currently known, so there is no well-defined group action on multiple zeta values. This   can be circumvented using 
 motivic multiple zeta values.  The  action of $G^{dR}$ on the de Rham fundamental group of $X$ can then be studied via its action on these objects.

\subsection{Motivic periods}
Let $T$ be a neutral Tannakian category over $\QQ$  with two fiber functors
$\omega_B , \omega_{dR} : T \rightarrow \mathrm{Vec}_{\QQ}$.
Define the ring of motivic periods to be the affine ring of functions on the scheme of  tensor isomorphisms from $\omega_{dR}$ to $\omega_B$
$$\Pe_T^{\mm} = \Or( \Isom_T(\omega_{dR}, \omega_B)) \ .  $$
Every motivic period can be constructed from  an object $M \in T$, and a pair of elements $w \in \omega_{dR}(M),  \sigma \in \omega_{B}(M)^{\vee}$. Its matrix coefficient is the function
$$\phi \mapsto   \langle \phi(w), \sigma \rangle \quad : \quad  \Isom_T(\omega_{dR}, \omega_B)  \rightarrow  \mathbb{A}^1_{\QQ}$$
where $\mathbb{A}^1_{\QQ}$ is the affine line over $\QQ$,  and defines an element denoted $[M, w, \sigma]^{\mm} \in \Pe_T^{\mm}$.   It is straightforward to write down linear relations between these symbols as well as a formula for the product of two such symbols.
 If,  furthermore,  there is an element $\mathrm{comp}_{B, dR} \in \Isom_T(\omega_{dR}, \omega_B)(\CC)$
 we can pair with it to get a map
\begin{equation}\label{permap} 
\per : \Pe_T^{\mm} \longrightarrow \CC
\end{equation}
called the period homomorphism.   The ring $\Pe_T^{\mm}$ admits a left action of the group $G^{dR} = \Isom_T(\omega_{dR}, \omega_{dR})$, or equivalently, a left coaction
\begin{equation} \label{Gdrcoact}
\Pe_T^{\mm} \longrightarrow \Or(G^{dR}) \otimes_{\QQ} \Pe_T^{\mm}\ .
\end{equation} 
\begin{example} \label{exampleLef} Let $T$ be any category of mixed Tate motives over a number field. It contains the Lefschetz
motive $\Lef=\QQ(-1)$, which is the motive  $H^1(\Pro^1\backslash \{0,\infty\})$.  Its de Rham cohomology is the $\QQ$-vector space spanned by the class $[{dz \over z}]$
 and its Betti homology is spanned  by a small positive loop  $\gamma_0$ around $0$.   The Lefschetz motivic period is 
 $$\Lef^{\mm} =[ \Lef,[ \textstyle {dz \over z}] ,  [\gamma_0]]   \quad \in \quad  \Pe^{\mm}_T\    .  $$
Its period is $\per(\Lef^{\mm}) = 2 \pi i$. It transforms, under the rational points of the de Rham Galois group of $T$,   by $\Lef^{\mm} \mapsto \lambda \Lef^{\mm}$, for any $\lambda \in \QQ^\times$. 
\end{example}

This construction can be applied to any pair of fiber functors to obtain different notions of motivic periods. Indeed,
the elements of $\Or(G^{dR})$ can be viewed as `de Rham' motivic periods, or matrix coefficients of the form $[M, w,v]^{dR}$, where $w\in \omega_{dR}(M)$
and $v\in \omega_{dR}(M)^{\vee}$ (called framings).  Whenever the fiber functors  carry extra structures (such as `complex conjugation' on $\omega_B$ or a  `weight' grading on $\omega_{dR}$), then the ring of motivic periods inherits similar structures.

\subsubsection{Motivic MZV's} Let $T= \MT(\ZZ)$.  The Betti and de Rham realisations provide
functors $\omega_B, \omega_{dR}$, and integration defines a canonical element $\mathrm{comp}_{B,dR} \in \Isom_T(\omega_{dR}, \omega_B)(\CC)$. 
Since the de Rham functor $\omega_{dR}$ is graded by the weight, the ring of motivic periods $\Pe^{\mm}_{\MT(\ZZ)}$ is also graded.\footnote{In the field of multiple zeta values, the `weight' refers to  one half of the Hodge-theoretic weight, so that $\Lef^{\mm}$ has degree $1$ instead of $2$. I shall adopt this terminology from here on.}
It contains graded subrings
$$\Pe^{\mm,+}_{\MT(\ZZ),\RR} \  \subset \  \Pe^{\mm,+}_{\MT(\ZZ)} \  \subset  \  \Pe^{\mm}_{\MT(\ZZ)}   $$ 
of  geometric periods (periods of motives whose weights are $\geq 0$),  denoted by  a superscript $+$, and those which  
 are also invariant under complex conjugation (denoted by a subscript $\RR$, since their   periods lie in $\RR$ as opposed to $\CC$).

Next, one has to show that the integral  $(\ref{MZVasitint})$ defines a period of an object $M$ in $\MT(\ZZ)$ (this can be done in several ways:  \cite{GM}, \cite{Terasoma}, \cite{DG})).  
This defines a matrix coefficient $[M, w, \sigma]^{\mm}$,  where $w$ encodes the integrand, and  $\sigma$ the domain of integration, which we call  a motivic multiple zeta value (\S \ref{sectMotMZV})  $$\zetam(n_1,\ldots, n_r) \in \Pe^{\mm}_{\MT(\ZZ)}\ . $$
Its weight is $n_1+\ldots +n_r$ and its   period is  $(\ref{MZVasitint})$.  Most  (but not all) of the known algebraic relations between multiple zeta values
are also known to hold for their motivic versions. Motivic multiple zeta values generate a  graded subalgebra
 \begin{equation} \label{Hdef} 
 \Ho \quad  \subset \quad  \Pe^{\mm,+}_{\MT(\ZZ),\RR}\  . 
 \end{equation}
   The description $\S\ref{sectIntroUnip}$, (1) of  $U^{dR}$ 
 enables
 one to compute the dimensions of the motivic periods of $\MT(\ZZ)$ in each degree by a simple counting argument:
 \begin{equation}  \label{dNdef} 
\hbox{if} \quad d_N:= \dim_{\QQ} \big( \Pe^{\mm,+}_{\MT(\ZZ),\RR}\big)_N   \quad \hbox{ then } \quad \sum_{N\geq 0} d_N  t^N= { 1 \over 1-t^2 -t^3}  \  .
 \end{equation} 
  This implies a theorem proved independently by Goncharov and Terasoma \cite{DG},\cite{Terasoma}.

\begin{theorem} The $\QQ$-vector space spanned by  multiple zeta values of weight $N$ has dimension at most 
$d_N$, where the integers $d_N$ are defined in $(\ref{dNdef})$.
 \end{theorem} 

So far, this does not use the action of the motivic Galois group, only a bound on the size of the motivic periods of $\MT(\ZZ)$. The role of  $\Pro^1\backslash \{0,1,\infty\}$  is that the 
  automorphism group  of its   fundamental groupoid  yields 
  a  formula for the coaction $(\ref{Gdrcoact})$ on the motivic multiple zeta values  (\S\ref{sectDualformula}). 
 The main theorem   uses this coaction in an essential way, and is  inspired by a conjecture of M. Hoffman \cite{Hoff}.

\begin{theorem} \label{thmHoffMZVLi} The  following set of motivic MZV's are linearly independent:
\begin{equation}\label{HoffmotMZVs} 
 \{ \zetam(n_1,\ldots, n_r)  \quad \hbox{ for } \quad  n_i = \{2,3\}\}\ . 
\end{equation}
\end{theorem} 
From the enumeration $(\ref{dNdef})$  of the dimensions, we deduce   that $\Ho=  \Pe^{\mm,+}_{\MT(\ZZ),\RR}$, and that    $(\ref{HoffmotMZVs})$
is a basis for 
$\Ho$. From this,  one immediately sees that $U^{dR}$ acts faithfully on $\Ho$, and 
theorem $\ref{thmfaithfulactionofGdr}$ follows easily.
As a bonus we obtain that $U^{dR}$ has canonical generators $\sigma_{2n+1}$ (defined in \S\ref{sectCanGen}), and, furthermore, by applying the period map we obtain the 

\begin{corollary} Every multiple zeta value of weight $N$ is a $\QQ$-linear combination of $\zeta(n_1,\ldots, n_r)$, where $n_i \in \{2,3\}$ and $n_1+\ldots+n_r=N$.
\end{corollary}

The point of  motivic periods is that they  give a mechanism for obtaining information on the action of $G^{dR}$, via the period map, from  arithmetic
relations between real numbers. For theorem $\ref{thmHoffMZVLi}$, the required arithmetic information  comes from a formula
for $\zeta(2,\ldots, 2, 3,2, \ldots, 2)$ proved by Zagier  \cite{Zagier} using analytic techniques.

\subsection{Transcendence of motivic periods} \label{sectTransperiods}
With hindsight, theorem   \ref{thmHoffMZVLi} has less to do with mixed Tate motives, or indeed $\Pro^1\backslash \{0,1,\infty\}$, than one might think.
Define a category $H$ whose objects are given by the following data:
\begin{enumerate}
\item A finite-dimensional $\QQ$-vector space $M_B$ equipped with an  increasing filtration called the weight, which is  denoted by $W$.
\item  A finite-dimensional $\QQ$-vector space $M_{dR}$ equipped with an increasing filtration $W$ and a  decreasing filtration $F$ (the Hodge filtration).
\item  An  isomorphism
$ \mathrm{comp}_{B,dR} : M_{dR} \otimes \CC \overset{\sim}{\rightarrow} M_{B}\otimes \CC$
which respects the weight filtrations.  The vector space $M_B$, equipped with $W$ and the  filtration $F$ on $M_B\otimes \CC$ induced by 
$ \mathrm{comp}_{B,dR} $ is a $\QQ$-mixed Hodge structure.
 \end{enumerate}
The category $H$ is Tannakian (\cite{DeP1}, 1.10), with two fiber functors, so it has a ring of motivic periods  $\Pe_{H}^{\mm}$. Furthermore, the Betti and de Rham realisations define a  functor $M  \mapsto     (M_{B}, M_{dR}, \mathrm{comp}_{B,dR}) :\MT(\ZZ) \rightarrow H$, and hence 
a homomorphism
\begin{align} \label{PeMTZtoH} 
  \Pe_{\MT(\ZZ)}^{\mm}   &  \longrightarrow  \Pe_{H}^{\mm}\ . 
\end{align} 
This map is known to be injective, but we do not need this fact.
The main theorem  \ref{thmHoffMZVLi}  is equivalent to saying that the images   $\zeta^H(n_1,\ldots, n_r) \in  \Pe_{H}^{\mm}$ of  $(\ref{HoffmotMZVs})$  for $n_i\in \{2,3\}$     are linearly independent. 
 In this way, we could have dispensed with motives altogether and worked with objects in $\Pe_{H}^{\mm}$, which are elementary.\footnote{In fact, we should never need to compute relations explicitly  using `standard operations' such as those described in \cite{KoZa}; these are taken
 care of automatically by the Tannakian formalism, and the bound on the Ext groups of $\MT(\ZZ)$ coming from Borel's theorems on  algebraic $K$-theory.} 
  This leads to the following philosophy for a theory of transcendence of motivic periods in $H$ (or another  suitable category of mixed Hodge structures). It differs from standard approaches which emphasise finding  relations between periods \cite{KoZa}.
\begin{itemize}
\item Write down arithmetically interesting elements in, say  $\Pe_{H}^{\mm}$, which come from geometry (i.e., which are periods in the sense of \cite{KoZa}).
\item Compute the coaction $(\ref{Gdrcoact})$ on these motivic periods, and use it to prove  algebraic independence theorems.
\end{itemize}
Indeed, there is no reason to restrict oneself  to mixed Tate objects, as the category $H$ does not rely on any conjectural properties of mixed motives.
The role of  $\Pro^1\backslash \{0,1,\infty\}$ was  to give an integral representation for the  numbers $(\ref{MZVasitint})$ and provide a formula for the coaction.

\subsubsection{Multiple modular values} Therefore, in the final part of this talk I want to propose changing the underlying geometry altogether, and replace 
a punctured projective line with (an orbifold) $M=\Gamma \bq \HH$, where $\HH$ is the upper half plane, and 
 $\Gamma \leq \mathrm{SL}_2(\ZZ)$ is a subgroup of finite index.  Because of Belyi's theorem \ref{thmBelyi},  every smooth  connected  projective algebraic curve over a number field is isomorphic to an $\overline{\Gamma \backslash \HH}$.
 Therefore the  (pure) motivic periods  obtained in this way   are   extremely rich\footnote{Grothendieck refers to $\mathrm{SL}_2(\ZZ)$ as   `\emph{une machine \`a motifs'}}.  It is reasonable to hope that the action of the Tannaka group on the  \emph{mixed} motivic periods of $M$ should be correspondingly rich and should generate  a  large class of new periods suitable for applications in arithmetic and  theoretical physics.
 Many of these  periods can be obtained as regularised iterated integrals on $M=\Gamma \backslash \HH$ (building on those considered by Manin in \cite{Ma1,Ma2}), and the philosophy of \S\ref{sectTransperiods} concerning their Galois action can be carried out  by computing a suitable automorphism group of non-abelian group cocyles. There still remains a considerable amount of work to put this general programme in its proper motivic context and extract all the arithmetic consequences.

\subsection{Contents}  In \S\ref{sect2}, I review the motivic fundamental group of $X$ from its Betti and de Rham view points, define motivic multiple zeta values, and derive their Galois action   from first principles.
The only novelty  is a direct derivation of the infinitesimal coaction from Ihara's formula.
 In \S\ref{sectMain}, I state some consequences of theorem \ref{thmHoffMZVLi}.
 In \S\ref{sectRoots} I  explain some results of Deligne concerning the motivic fundamental group of the projective line minus $N^\mathrm{th}$ roots of unity, 
 and in \S\ref{sectDepth} discuss the depth filtration on motivic multiple zeta values and its conjectural connection with modular forms.
 In  \S\ref{sectMMV} I mention some new results on multiple modular values for $\mathrm{SL}_2(\ZZ)$, which forms a bridge between multiple zeta values and modular forms.
 
 For reasons of space, it was unfortunately not possible to review the large recent  body of work relating to associators, double shuffle equations 
 (\cite{An} \S25,   \cite{Fu}, \cite{Racinet})  and applications to knot theory, the Kashiwara-Vergne problem, and related topics such as deformation quantization; let alone 
 the vast range of  applications of multiple zeta values to high-energy physics and string theory.
 Furthermore, there has been recent progress in $p$-adic aspects of multiple zeta values, notably by H. Furusho and G. Yamashita, and work of M. Kim on integral points and the unit equation, which is
 also beyond the scope of these notes.

Many  technical aspects  of mixed Tate motives have also been omitted. See \cite{DG}, \S1-2 for the definitive reference.

\section{The motivic fundamental group of $\Pro^1 \backslash \{0,1,\infty\}$}\label{sect2}
Let $X= \Pro^1 \backslash \{0,1,\infty\}$, and for now let $x,y  \in X(\CC)$.
The motivic fundamental groupoid of $X$ (or rather, its Hodge realisation) consists of the following data:
\begin{enumerate}
\item (Betti).  A collection of schemes $\pi_1^B(X,x,y)$ which are  defined over $\QQ$, and which are equipped with the structure of a groupoid:
$$ \pi_1^B(X,x,y) \times \pi_1^B(X,y,z) \longrightarrow \pi_1^B(X,x,z)$$
for any $x,y,z \in X(\CC)$. There is a natural homomorphism
\begin{equation}  \label{pitop2piB}
 \pi^{top}_1(X,x,y) \longrightarrow \pi_1^B(X,x,y)(\QQ) 
 \end{equation}
where the fundamental  groupoid on the left is given by homotopy classes of paths relative to their endpoints. The previous map is Zariski dense.
\item (de Rham).  An affine group scheme\footnote{ It  shall also be  written
$\pi_1^{dR}(X,x,y)$ but  does not depend on the choice of base points. The fact that there is a canonical isomorphism $\pi_1^{dR}(X,x,y) = \pi_1^{dR}(X)$  is 
equivalent to saying that there is a `canonical de Rham path' between the  points $x$ and $y$.}   over $\QQ$  denoted by $\pi_1^{dR}(X)$. 

\item (Comparison). A canonical isomorphism of schemes over $\CC$
\begin{equation}\label{comparisonIsom3}
\mathrm{comp}_{B,dR}:  \pi_1^{B}(X,x,y)\times_{\QQ} \CC \overset{\sim}{\longrightarrow} \pi_1^{dR}(X)\times_{\QQ} \CC\ .
\end{equation}
\end{enumerate}
These structures are described below.
Deligne has explained (\cite{DeP1}, \S15) how to replace ordinary base points with tangential base points in various settings. Denote such a tangent vector  by 
$${\overset{\rightarrow}{v}\!}_x = \hbox{the tangent vector } v \in T_{x}(\Pro^1(\CC))  \hbox{ at the point } x\ .$$
Identifying $T_{x}(\Pro^1(\CC))$ with $\CC$, one obtains  natural tangent vectors $\tone_0$ and $-\tone_1$ at the points $0$ and $1$ respectively, and a canonical path, or `droit chemin' 
$$\dch  \in   \pi^{top}_1(X,\tone_0,-\tone_1)$$
given by the straight line which travels  from $0$ to $1$  in $\RR$  with unit speed. 

The reason for taking the above tangential base points is to ensure that the corresponding motive (theorem \ref{thmpi1ismotivic}) has good reduction modulo all primes $p$: in the setting of $\Pro^1\backslash \{0,1,\infty\}$ there are no ordinary base points with this property.

The following theorem states that the structures $1-3$ are motivic.
\begin{theorem} \label{thmpi1ismotivic} There is an ind-object  (direct limit of objects)
\begin{equation}\label{pi1motivic} 
\Or( \pi_1^{mot} ( X, \tone_0, -\tone_1 )) \in \mathrm{Ind}\,  (\MT(\ZZ)) 
\end{equation}
whose Betti and de Rham realisations are the affine rings $\Or(\pi_1^B(X,\tone_0, -\tone_1))$, and $\Or(\pi_1^{dR}(X))$, respectively. \end{theorem}
\begin{proof} (Sketch)
The  essential idea is due to Beilinson  (\cite{GoMTM}, theorem 4.1) and Wojtkowiak \cite{Wo}.  Suppose, for simplicity, that $M$ is a connected manifold and $x,y \in M$ are distinct points. Consider the submanifolds in $M\times \ldots \times M$ ($n$ factors):
$$ N_i = M^{i-1} \times \Delta \times M^{n-i-1}  \qquad \hbox{ for } i=1, \ldots, n-1$$
where $\Delta$ is the diagonal $M \subset M\times M$. Set $N_0 = \{x\} \times M^{n-1}$ and $N_{n} =  M^{n-1} \times \{y\} $, and  let $N\subset M^n$ be the union of the $N_i$, for $i=0,\ldots, n$. Then 
\begin{equation} \label{HkMnN} 
 H_k (M^{n}, N) = \begin{cases} \QQ[ \pi_1^{top}(M,x,y)] / I^{n+1} \quad   \hbox{ if } k = n  \\  0 \qquad  \qquad \qquad \qquad \qquad \hbox{ if } k< n \end{cases}
 \end{equation} 
where the first line is the $n^\mathrm{th}$ unipotent truncation of the fundamental torsor of paths from $x$ to $y$ ($I$ is the image of the augmentation 
ideal in $\QQ[ \pi_1^{top}(M,x)] $;  see below). In the case when $M=\Pro^1 \backslash \{0,1,\infty\}$, the left-hand side  of $(\ref{HkMnN})$ defines a mixed Tate motive.  The case when $x=y$, or when $x$ or $y$ are tangential base points, is more delicate \cite{DG}, \S3.
\end{proof}

The Betti and de Rham realisations can be described  concretely as follows.
\begin{enumerate}
\item (Betti). The Betti fundamental groupoid is defined to be the pro-unipotent completion of the ordinary topological fundamental groupoid.
 For simplicity, take $x=y \in X(\CC)$.  Then there is an exact sequence
 $$0 \longrightarrow I \longrightarrow \QQ[\pi_1^{top}(X(\CC), x) ] \longrightarrow \QQ \longrightarrow 0 $$
 where the third map sends the homotopy class of any path $\gamma$ to $1$ (thus $I$ is the augmentation ideal). Then one has (Mal\v{c}ev, Quillen)
 $$ \Or( \pi_1^B(X,x) ) = \lim_{N \rightarrow \infty} \Big( \QQ[ \pi_1^{top}(X,x)] / I^{N+1}\Big)^{\vee} $$
 The case when $x\neq y$ is defined in a  similar way, since $\QQ[\pi_1^{top}(X(\CC), x,y) ]$ is a rank one module over  $\QQ[\pi_1^{top}(X(\CC), x) ] $.
\item (de Rham). When $X = \Pro^1\backslash \{0,1,\infty\}$, one verifies that 
$$\Or(\pi_1^{dR}(X)) \cong \bigoplus_{n \geq 0} H^1_{dR}(X)^{\otimes n} $$
which is isomorphic to the tensor coalgebra on the two-dimensional graded $\QQ$-vector space $H^1_{dR}(X)\cong \QQ(-1) \oplus \QQ(-1)$. We can take as  basis  the elements
$$[\omega_{i_1} | \ldots | \omega_{i_n}]  \quad \hbox{ where } \omega_{i_k} \in  \textstyle{\{ {dt \over t}, {dt \over 1-t}\}}$$
where the bar notation denotes a tensor product $\omega_{i_1}\otimes \ldots \otimes \omega_{i_n}$. It is a Hopf algebra for  the shuffle product and deconcatenation coproduct
and is graded in degrees $\geq 0$ by the degree which assigns  ${dt \over t}$ and ${dt \over 1-t}$  degree $1$.
\end{enumerate}
\vspace{-0.05in}
Denoting  $\tone_0$ and $-\tone_1$ by $0$ and $1$ respectively, 
one can  write, for $x, y \in \{0,1\}$
$${}_x\Pi^{\bullet}_y = \mathrm{Spec}( \Or(\pi_1^{\bullet}(X,x,y))   \qquad \hbox{ where } \quad \bullet \in  \{B, dR, \mathrm{mot}\}\ .$$
It is convenient to write ${}_x\Pi_y$ instead of  ${}_x\Pi^{dR}_y$.
It does not depend on $x$ or $y$, 
but admits an action of the motivic Galois group $G^{dR}$ which 
is  sensitive to $x$ and $y$.  If $R$ is any commutative unitary $\QQ$-algebra,
$${}_x\Pi_y (R)  \cong \{ S \in R\langle\langle x_0, x_1 \rangle \rangle^{\times}  : \Delta S= S \otimes S \}$$
is isomorphic to the group of invertible formal power series in two non-commuting variables $x_0, x_1$, which are group-like for the completed
coproduct $\Delta$ defined by  $\Delta(x_i) = x_i \otimes 1 + 1 \otimes x_i$.  
The group law is given by concatenation of series.

\subsection{Periods}  \label{Periodsofpi1} The periods of the motivic fundamental groupoid of $\Pro^1\backslash \{0,1,\infty\}$ are the coefficients of the 
comparison isomorphism $\mathrm{comp}_{B,dR}$ $(\ref{comparisonIsom3})$ with respect to the $\QQ$-structures on the Betti and de Rham sides.
 Let 
$${}_01_1^B \quad   \in \quad  \pi_1^B(X,\tone_0, -\tone_1) (\QQ)   \quad \subset  \quad \Or( \pi_1^B(X,{\small \tone_0 , -\tone_1}) ) ^{\vee}$$
denote the image of $\dch$ under the natural map $(\ref{pitop2piB})$. It should be viewed
as a linear form on  the affine ring  of the Betti $\pi_1$.
For all $\omega_{i_k} \in \{ {dt \over t}, {dt \over 1-t} \}$, 
\begin{equation} \label{genitint}
\langle \mathrm{comp}_{B,dR}([\omega_{i_1}| \ldots | \omega_{i_n}]),  {}_01_1^B \rangle = \int_{\dch} \omega_{i_1} \ldots   \omega_{i_n} \end{equation}
The right-hand side is the  iterated integral  from $0$ to $1$,  \emph{regularised} with respect to the tangent vectors $1$ and $-1$ respectively,
of the one-forms $\omega_{i_k}$.  No regularisation is necessary in the case when $\omega_{i_1} = {dt \over 1-t}$ and $\omega_{i_n} = {dt \over t}$,
and in this case the right-hand side reduces to the formula $(\ref{MZVasitint})$.  In general, one can easily show: 
\begin{lemma} The  integrals $(\ref{genitint})$ are $\ZZ$-linear combinations of MZV's of weight $n$.
\end{lemma} 
The \emph{Drinfeld associator}  is the de Rham image of $\dch$
$$\mathcal{Z}=  \mathrm{comp}_{B, dR}({}_01_1^B) \in \opi(\CC)$$
 Explicitly, it is  the non-commutative generating series of the integrals $(\ref{genitint})$
\begin{align}
 \mathcal{Z} &= \sum_{i_k \in \{0,1\} }  x_{i_1} \ldots x_{i_n}   \int_{\dch} \omega_{i_1} \ldots   \omega_{i_n}  \\
 & = 1 + \zeta(2) [x_1, x_0]+ \zeta(3) ([x_0,[x_0,x_1]] + [x_1,[x_1,x_0]]  ) +\cdots 
\end{align}
It is an exponential of a Lie series. 
\subsection{Motivic multiple zeta values} \label{sectMotMZV}
By the previous paragraph,  the affine ring of the de Rham  fundamental group is the graded Hopf algebra
$$\Or({}_x \Pi_y) \cong \QQ \langle e_0, e_1 \rangle$$
independently of $x,y \in \{0,1\}$.  Its product is the shuffle product, and its coproduct is deconcatenation. Its basis elements can be indexed by words in $\{0,1\}$. By a general fact about shuffle algebras, the antipode is the 
map
$w \mapsto w^*$  where  $$(a_1\ldots a_n)^* = (-1)^n a_n \ldots a_1$$
is signed reversal of words. Thus any word $w$ in $\{0,1\}$ defines a de Rham element in $\Or({}_x \Pi_y)$.  The augmentation map $\QQ\langle e_0,e_1\rangle \rightarrow \QQ$ corresponds to the unit element in the de Rham fundamental group and 
defines a linear form ${}_x1^{dR}_y \in  \Or({}_x \Pi_y)^{\vee}$. 

Define  Betti linear forms ${}_x1_y^B \in \Or( {}_x \Pi^B_y)^{\vee} $ to be the images of the  paths 
$$ 
   \dch  \hbox{ if } x=0, y=1 \quad ; \quad 
   \dch^{-1}    \hbox{ if }  y=1, x=0  \quad ; \quad 
   c_x   \hbox{ if } x=y \ ,
 $$
 where $\dch$ is the straight path from $0$ to $1$, $\dch^{-1}$ is the reversed path from $1$ to $0$, and $c_x$ is the constant 
 (trivial) path based at $x$.
 
 Out of this data we can construct the following  motivic  periods.

\begin{definition} Let $x,y\in \{0,1\}$ and let $w$ be any word in $\{0,1\}$.  Let
\begin{equation} 
I^{\mm}(x;w;y)  = [  \Or({}_x\Pi_y^{\mathrm{mot}}),  w,  {}_x1_y^B]^{\mm}  \qquad \in \quad \Pe^{\mm,+}_{\MT(\ZZ), \RR} 
\end{equation}
\end{definition} 
We call the elements $I^{\mm}$ motivic iterated integrals. 
 The   `de Rham' motivic period is the matrix coefficient  $ [ \Or({}_x\Pi_y^{\mathrm{mot}}),  w,  {}_x1_y^{dR}]$  on $\MT(\ZZ)$ with respect to the fiber functors $\omega_{dR}, \omega_{dR}$.
It defines a function on  $G^{dR}$. Its   restriction  to the prounipotent group $U^{dR}$ defines an element $I^{\uu}(x;w;y) \in \Or(U^{dR})$. 
The latter are  equivalent to objects  defined by Goncharov  (which he also called motivic iterated integrals).
\begin{defn} Define motivic (resp. unipotent) multiple zeta values by 
\begin{equation}
\zeta^{\bullet} (n_1,\ldots, n_r) = I^{\bullet} (0; 1 0^{n_1-1}  \cdots 1 0^{n_r-1} ; 1) \  ,  \quad \bullet = \mm, \uu \\  \nonumber 
\end{equation}
\end{defn}
It is important to note that $\zetam(2)$ is non-zero, whereas $\zeta^{\uu}(2)=0$.\footnote{One can define a homomorphism 
$\Pe^{\mm,+}_{\MT(\ZZ),\RR} \rightarrow \Pe^{\uu}_{\MT(\ZZ)}$ which sends $\zetam(n_1,\ldots, n_r)$ to $\zeta^{\uu}(n_1,\ldots, n_r)$ and prove that its kernel is the ideal generated by $\zetam(2)$.}
We immediately deduce from the definitions that
\begin{align} \begin{split}  \label{Iwproperties}
(i). & \quad I^{\mm} (x;w;x)   = \delta_{w, \emptyset}  \qquad \hbox{ for } x  \in \{0,1\}   \\
(ii). & \quad I^{\mm} (x;w;y)  = I^{\mm}(y;w^*;x)     \end{split}
\end{align} 
The first property holds because the constant path is trivial, the second follows from the antipode formula and because $\dch \circ \dch^{-1}$, or $\dch^{-1} \circ \dch$,   is  homotopic to a constant path. Finally, replacing multiple zeta values with their motivic versions, we can define a  motivic version of  the Drinfeld associator
\begin{equation} \label{motivicDrinfeld}
\mathcal{Z}^{\mm} = \sum_{i_1,\ldots, i_n \in \{0,1\}} x_{i_1}\ldots x_{i_n} I^{\mm}(0; i_1, \ldots,  i_n;1)  \ .
\end{equation} 
It satisfies the associator equations defined by Drinfeld \cite{Drinfeld}, on replacing $2 \pi i $ by $\Lef^{\mm}$ (using the fact that $\zetam(2) = {-(\Lef^{\mm})^2 \over 24} $),
and the double shuffle equations of \cite{Racinet}.

\subsection{Action of the motivic Galois group} \label{sectActUdr}
The category $\MT(\ZZ)$ is a Tannakian category with respect to the de Rham fiber functor.  Therefore the motivic Galois 
group acts on the affine ring $\Or(\opi)$ of the de Rham realisation of the motivic fundamental torsor of path $(\ref{pi1motivic})$. A slight generalisation of theorem \ref{thmpi1ismotivic}
shows that $G^{dR}$ acts on the de Rham fundamental schemes
$${}_x \Pi_y \qquad \hbox{ for all } x, y \in \{0,1\} $$
and furthermore, is compatible with the following structures:
\begin{itemize}
\item (Groupoid structure). The multiplication maps
$${}_x \Pi_y \times {}_y \Pi_z \longrightarrow   {}_x \Pi_z $$
for all $x,y,z \in \{0,1\}$. 
\item (Inertia). The action of $U^{dR}$ fixes the elements  
$$\exp(x_0) \hbox{ in } {}_0 \Pi_0(\QQ) \qquad \hbox{ and } \qquad  \exp(x_1) \hbox{ in } {}_1 \Pi_1(\QQ)
$$
\end{itemize}

The groupoid structure is depicted in   figure 1.

\begin{figure}[h!]
  \begin{center}
   \epsfxsize=4cm \epsfbox{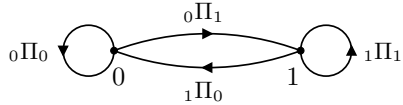}
     \put(-93,-3){$0$}  \put(-66,22){\small ${}_0 \Pi_1$}
        \put(-27,-3){$1$}  \put(-66,-8){\small ${}_1 \Pi_0$}
         \put(-132,7){\small ${}_0 \Pi_0$}
     \put(2,7){\small ${ {}_1 \Pi_1}$}
 \caption{The groupoid ${}_x\Pi_y$ for $x, y\in \{0,1\}$. The diagram only represents the groupoid structure; the paths shown do not accurately depict  the tangential base points.}
  \end{center}
\end{figure}

The  local monodromy map 
$ \pi_1^{top}(\GG_m, \tone_0) \rightarrow  \pi_1^{top}(X, \tone_0) $  (where we write  $\GG_m$ for  $\Pro^1\backslash \{0,\infty\}$),   corresponding to monodromy around $0$, 
 has a motivic analogue which gives rise to  the inertial condition. Its de Rham realisation is the map
 $$ \pi_1^{dR}(\GG_m, \tone_0)  \rightarrow   \pi_1^{dR}(X, \tone_0)=  {}_0 \Pi_{0}$$  
and is respected by $G^{dR}$. One shows that  $U^{dR}$ acts trivially on $ \pi_1^{dR}(\GG_m, \tone_0)$, 
and furthermore that the element $\exp(x_0) \in {}_0 \Pi_0(\QQ)$ is in the image of the previous map. This gives the first inertial condition. 
\begin{remark} It is astonishing that one obtains much  useful information at all from such symmetry considerations.  Nonetheless, it is enough to show the faithfulness of the action of $G^{dR}$ (below).  There are further structures respected by $G^{dR}$, such as  compatibilities with automorphisms of $\Pro^1 \backslash \{0,1,\infty\}$. They are not required.
\end{remark}

\subsection{Ihara action} \label{sectIharaact}
Let $\Au$ denote the group of automorphisms of the groupoid ${}_x \Pi_{y}$ for $x,y\in \{0,1\}$ which respects the structures
$1,2$ described in \S\ref{sectActUdr}. 
\begin{proposition} The scheme $\opi$ is an $\Au$-torsor. In particular, the action of $\Au$ on $1 \in \opi$ defines   an isomorphism of schemes
\begin{equation}  \label{Auactson1}
a \mapsto a(1) : \Au  \longrightarrow  \opi\ .
\end{equation}
The action  of $\Au$  on $\opi$  defines, via this isomorphism,  a new group law $$\circ: \opi \times \opi \rightarrow \opi \ . $$ 
 It is given explicitly on   formal power series by  Ihara's formula
\begin{align}  \label{Iharaaction}
A(x_0,x_1) \circ G(x_0,x_1) & = G(x_0, A x_1 A^{-1}) A 
\end{align}
\end{proposition}
\begin{proof}
For the basic geometric idea, see \cite{IharaICM}, \S 2.3. Let $\aaa \in \Au$, and write $a_{xy}(\xi)$ for the action of $\aaa$
on $\xi\in {}_x \Pi_y$. Write $a = a_{01} (1)$. Since ${}_0 \Pi_0$ is a group, $\aaa$ acts trivially on its identity element, and so
$a_{00}(1)=1$. Via the map $ {}_0 \Pi_{1} \times {}_1 \Pi_{0}  \rightarrow {}_0 \Pi_{0}$ we have 
$a_{01}(1) a_{10}(1) = a_{00}(1) $
and hence $a_{10}(1) = a^{-1}$.  The inertial conditions give
\begin{equation}\label{Autinert} a_{00} (\exp(x_0)) = \exp(x_0) \quad \hbox{ and }  \quad a_{11} (\exp(x_1))= \exp(x_1)
\end{equation} 
 Now the composition of paths
$ {}_1 \Pi_{0} \times {}_0 \Pi_{0} \times {}_0 \Pi_{1} \rightarrow {}_1 \Pi_{1}$
gives rise to an equation $1. \exp(x_1). 1 = \exp(x_1)$. Applying $\aaa$ to this gives by the second equation in $(\ref{Autinert})$
\begin{equation} \label{Aooonexpx1} 
a_{00}( \exp(x_1)) = a \exp(x_1) a^{-1}  = \exp( a x_1 a^{-1}) 
\end{equation} 
which completely determines the action of $\Au$ on ${}_0 \Pi_0$. Via the map $ {}_0 \Pi_{0} \times {}_0 \Pi_{1}  \rightarrow {}_0 \Pi_{1}$
we have the equation $g. 1 = g$, and hence 
\begin{equation}\label{a01froma00}
a_{01}(g)  = a_{00}(g) . a \ .
\end{equation}
Formula $(\ref{Iharaaction})$  follows from $(\ref{Autinert})$, $(\ref{Aooonexpx1})$, $(\ref{a01froma00})$. One easily checks that $a$ uniquely determines $\aaa$, and so  $(\ref{Auactson1})$ is an isomorphism (see also \cite{DG}, 5.9.)
\end{proof} 

The groupoid and inertia structures are preserved by $U^{dR}$, giving  a morphism
\begin{equation} \label{UtoAu}
 \rho: U^{dR} \longrightarrow \Au \overset{(\ref{Auactson1})}{\cong} \opi
 \end{equation} 
such that the following diagram commutes
\begin{align}  \label{udrcommutativediagram}
 U^{dR}  \times \opi&    \longrightarrow \opi    \\  
  { {}_{\rho\times \mathrm{id}}} \downarrow \qquad   &  \ \quad \quad \downarrow_{\mathrm{id}}  \nonumber \\
  \opi  \times \opi  & \overset{\circ}{\longrightarrow} \opi  \nonumber
\end{align}

In principle this describes the action of the motivic Galois group on $\opi$. Note, however, that the map $(\ref{UtoAu})$ is mysterious and very
 little is known about it.

\subsection{Dual formula}  \label{sectDualformula} The coaction on motivic iterated integrals is  dual to Ihara's formula.
Dualising $(\ref{udrcommutativediagram})$, we have
$$\Delta: \Or(\opi) \longrightarrow \Or(U^{dR}) \otimes \Or(\opi)$$
It is equivalent, but more convenient, to consider the infinitesimal coaction 
$$D \ : \  \Or(\opi) \longrightarrow   \mathcal{L}  \otimes \Or(\opi)   \qquad \big( D(x) = \Delta(x) - 1\otimes x  \mod  \Or(U^{dR})_{>0} ^2 \big)$$
where $\mathcal{L} =  \Or (U^{dR})_{>0}/\big(\Or(U^{dR})_{>0} \big)^2  $ is the Lie coalgebra of indecomposables in $\Or(U^{dR})$. 
The following formula  is an infinitesimal variant of a formula due to Goncharov \cite{GG}, relating to slightly different objects. In order to fill a gap in the literature,
I will  sketch how it follows almost immediately from Ihara's formula.

 \begin{proposition} Let $a_0,\ldots, a_{n+1}\in \{0,1\}$.  The coaction $D$ is given by 
\begin{align}\label{mainformula}   
 D ( I^{\mm}(a_0;a_1,\ldots, a_n; a_{n+1}) )  = & \sum_{0\leq p<q \leq n} \big[  I^{\uu} (a_{p} ;a_{p+1},\ldots, a_{q}; a_{q+1}) \big] \\
  & \otimes   I^{\mm}(a_{0}; a_1, \ldots, a_{p}, a_{q+1}, \ldots ,  a_n ;a_{n+1}) \nonumber  \ .
  \end{align}
  where the  square brackets on the left denote  the map $[\,\,]: \Or(U^{dR})_{>0} \rightarrow \mathcal{L}$.
\end{proposition}
\begin{proof}
Denote the action of $\mathrm{Lie}\, \Au$ on $\mathrm{Lie}\, {}_0 \Pi_0$ by $\circ_0$.
By     $(\ref{Auactson1})$,   $\mathrm{Lie}\, \Au \cong \mathrm{Lie}\, \opi$  is the set of 
primitive elements in its (completed) universal enveloping algebra  which we  denote simply by
$\mathcal{U} (\opi)$.
By $(\ref{Autinert})$  and $(\ref{Aooonexpx1})$  we have 
$ a \circ_0 x_0 = 0$   and $a \circ_0 x_1 = a x_1 - x_1 a$. 
 The antipode on $\mathcal{U} (\opi)$ is  given by the signed reversal $*$.   Since  $a \in   \mathcal{U} (\opi)$ is   primitive,  $a=-a^*$ and also 
$$  a \circ_0 x_0 = 0 \qquad \hbox{ and } \qquad  a \circ_0 x_1 = a x_1  + x_1 a^*\ .$$
This extends to an action on $\mathcal{U} ( {}_0 \Pi_0)$ via  $a\circ_0 w_1w_2 = (a \circ_0 w_1)w_2+ w_1(a \circ_0 w_2)$.
Now consider the action $a \circ_0 \cdot $  on the following words. All terms  are omitted except those terms where $a$ or $a^*$ is inserted in-between the  two bold letters:
\begin{align}
a \, \circ_0 \,w_1  \mathbf{ x_0 x_0} w_2  & = \cdots \  + \  0  \ +  \ \cdots \nonumber \\ 
a \,\circ_0 \,w_1  \mathbf{ x_0 x_1} w_2  & = \cdots \ +\  w_1 \mathbf{ x_0  a x_1 } w_2 \ + \ \cdots \nonumber \\ 
a\, \circ_0\, w_1  \mathbf{ x_1 x_0} w_2  & = \cdots  \ + \  w_1 \mathbf{ x_1 a^* x_0 } w_2   \ + \ \cdots \nonumber \\ 
a \,\circ_0\, w_1  \mathbf{ x_1 x_1} w_2  & =  \cdots  \ + \ \underbrace{  w_1 \mathbf{ x_1 a x_1  }  w_2 + w_1 \mathbf{ x_1 a^* x_1  }  w_2 }_{0}  \ + \ \cdots \nonumber 
\end{align}  
These four equations are  dual to all but the first and last  terms in  $(\ref{mainformula})$, using the fact that $I^{\uu}(x;w;x)=0$ for $x=0,1$ (first and fourth lines),
and the fact that $I^{\uu}(1;w^*;0) = I^{\uu}(0;w;1)$ (third line).  
A straightforward modification of the above argument taking into account the initial and final terms (using $(\ref{a01froma00})$)  shows   that the action $\circ_1$ of $\mathrm{Lie}\, \Au$ on $\opi$
is dual to the full expression $(\ref{mainformula})$.
\end{proof}
Armed with this formula, we  immediately deduce 
that for all $n\geq 2$, 
\begin{align} \label{zetamprimitives}
D\,  \zetam(n) & =   [\zeta^{\uu}(n)] \otimes 1  
\end{align}
where  we recall that $\zeta^{\uu}(2n)=0$. One easily shows that $\zeta^{\uu}(2n+1) \neq 0$ for $n\geq 1$.   See also \cite{HaMa}.
Denote the map $ w \mapsto [I^{\uu}(0;w;1)]: \Or(\opi)_{>0} \rightarrow \mathcal{L}$ simply by 
$\xi \mapsto [\xi^{\uu}]$. 
 From the structure $\S\ref{sectIntroUnip},1$ of $G^{dR}$ we have the following  converse to  $(\ref{zetamprimitives})$ (\cite{BMTZ},  \S3.2). 
\begin{theorem} \label{thmprimitives}
An  element $\xi \in \Or(\opi)$ of weight $n \geq 2 $  satisfies $D\xi=[\xi^{\uu}]\otimes 1$   if and only if $\xi \in \QQ\, \zetam(n)$.
\end{theorem} This theorem, combined  with   $(\ref{mainformula})$,  provides a powerful
method for proving identities between motivic multiple zeta values. Applications are given in \cite{BrDec}.  

\section{The main  theorem  and  consequences} \label{sectMain}
Theorem $\ref{thmHoffMZVLi}$ is a result about linear independence. There is an analogous statement for  algebraic independence 
of motivic multiple zeta values.
 
 \begin{defn} Let $X$ be an alphabet (a set) and let $X^{\times}$ denote the free associative monoid generated by $X$.  Suppose that $X$ has a total ordering  $<$, and  extend it to $X^{\times}$
 lexicographically. 
 An element $w\in X^{\times}$ is said to be a Lyndon word if 
 $$ w <u \quad  \hbox{ whenever } \quad  w = u v \quad  \hbox{ and }  \quad u, v \neq \emptyset\ .$$
 \end{defn}
 For an ordered set $X$, let $\mathrm{Lyn(X)}$ denote the set of Lyndon words in $X$.
 \begin{theorem} \label{thmAlgInd} Let $X_{3,2} = \{2,3\}$ with the ordering $3<2$. The set of elements 
 \begin{equation} \label{zetamHofflyndon} \zetam(w)\quad \hbox{ where } w \in \mathrm{Lyn}(X_{3,2}^{\times}) 
 \end{equation} 
are algebraically independent over $\QQ$, and generate the algebra $\Ho$ of motivic multiple zeta values.
 \end{theorem}

Theorem $\ref{thmAlgInd}$ implies that every motivic multiple zeta value is equal to a unique polynomial with rational coefficients in the elements
$(\ref{zetamHofflyndon})$.  It is often convenient to modify this generating family by  replacing $\zetam(3,2,\ldots, 2)$ (a three followed by $n-1$ two's) with $\zetam(2n+1)$
 (by theorem \ref{thmZagierthm}).
Taking the period yields the 

\begin{coro} Every multiple zeta value  is a polynomial, with coefficients in $\QQ$, in
\begin{equation} \label{ZHL} \zeta(w)\quad \hbox{ where } w \in \mathrm{Lyn}(X_{3,2}^{\times}) \ .
 \end{equation} 
\end{coro}

\begin{coro} The category $\MT(\ZZ)$ is generated by $\pi^{mot}_1(\Pro^1\backslash \{0,1,\infty\},\tone_0,-\tone_1)$ in the following sense. Every mixed Tate motive over $\ZZ$ is isomorphic, up to  a  Tate twist, to a direct sum of copies of sub-quotients of  
$$\Or(\pi^{mot}_1(\Pro^1\backslash \{0,1,\infty\}, \tone_0, -\tone_1))\ .$$
\end{coro} 

\begin{coro}
 The periods of mixed Tate motives over $\ZZ$ are polynomials  with rational coefficients of  $(2\pi i)^{-1}$ and  $(\ref{ZHL})$.
\end{coro}

More precisely \cite{DLetter}, if $M \in \MT(\ZZ)$ has non-negative  weights (i.e. $W_{-1}M=0$), then the periods of 
$M$  are polynomials   in $(\ref{ZHL})$ and $2 \pi i$.

\subsection{Canonical generators} \label{sectCanGen} Recall that the unipotent zeta values $\zeta^{\uu}$ are elements of 
$\Or(U^{dR})$.  As a consequence of theorem \ref{thmAlgInd}:

\begin{coro} For every $n\geq 1$ there is a canonical   element $\sigma_{2n+1}\in \mathrm{Lie} \, U^{dR}(\QQ)$  which is uniquely defined by  $\langle \exp(\sigma_{2n+1}), \zeta^{\uu}(2m+1) \rangle= \delta_{m,n} $, and 
\begin{align}
\langle \exp(\sigma_{2n+1}), \zeta^{\uu}(w) \rangle &  = 0 \qquad \hbox{ for all } w\in \mathrm{Lyn}(X_{3,2}) \hbox{ such that  } \deg_3 w>1 \  .\nonumber
\end{align}
\end{coro} 
The elements $\sigma_{2n+1}$ can be taken as generators in \S\ref{sectIntroUnip} (1).  
It is perhaps surprising that one can define canonical elements of the motivic Galois group at all. These should perhaps be taken with a  pinch
of salt, since there may be other natural generators for the algebra of motivic multiple zeta values.

\begin{coro} 
  There is a unique    homomorphism $\tau: \Ho \rightarrow \QQ$ (see $(\ref{Hdef})$) such that: 
$$\langle \tau, \zetam(2) \rangle  = - \textstyle{1 \over 24} $$
and $\langle \tau, \zetam(w) \rangle   = 0$ for all $w\in \mathrm{Lyn}(X_{3,2})$ such that $ w \neq 2$.
\end{coro}

Applying this map to the motivic Drinfeld associator defines a canonical  (but not explicit!) rational associator:
$$\tau (\mathcal{Z}^{\mm}) \in  \opi(\QQ) = \QQ \langle \langle  x_0, x_1 \rangle \rangle $$

By acting on the canonical rational associator with elements $\sigma_{2n+1}$, one  deduces that there exists  a huge space of rational associators (which forms a torsor
over $G^{dR}(\QQ)$). Such associators have several applications (see,  for example \cite{Fu}).
\subsection{Transcendence conjectures}

\begin{conjecture} A variant of Grothendieck's period conjecture  states that 
$$\per : \Pe_{\MT(\ZZ)}^{\mm} \longrightarrow \CC$$
is injective. In particular, its restriction to $\Ho$ is injective also.
\end{conjecture}
The last statement, together with theorem $\ref{thmHoffMZVLi}$, is equivalent to
\begin{conjecture} (Hoffman) The elements $\zeta(n_1,\ldots, n_r)$ for $n_i \in \{2,3\}$ are  a basis for the $\QQ$-vector space
spanned by multiple zeta values.
\end{conjecture} 
This in turn implies a  conjecture due to Zagier,  stating that the dimension of the $\QQ$-vector space of multiple zeta
values of weight $N$ is equal to $d_N$  $(\ref{dNdef}),$ and furthermore that the ring of multiple zeta values is graded by the weight.
Specialising further, we obtain  the  following  folklore
\begin{conjecture} The numbers $\pi, \zeta(3),\zeta(5), \zeta(7), \ldots $  are algebraically independent.\end{conjecture}

\subsection{Idea of proof of theorem $\ref{thmHoffMZVLi}$} The proof of linear independence is by induction on the number of $3$'s.  In the case where there are no 3's, one can easily  show  by adapting an argument due to Euler that
$$\zeta(\underbrace{2,\ldots, 2}_n) = {\pi^{2n}  \over (2n+1)!}\  . $$
The next interesting case is where there is one 3.
\begin{theorem} \label{thmZagierthm} (Zagier \cite{Zagier}).  Let $a,b\geq 0$. Then 
$$\zeta( \underbrace{2,\ldots 2}_a,3, \underbrace{2,\ldots, 2}_{b})=  2\,\sum_{r=1}^{a+b+1}(-1)^r (A^r_{a,b}-B^r_{a,b})\, \zeta(2r+1)\, \zeta(\underbrace{2,\ldots, 2}_{a+b+1-r})\,  $$
where, for any $a,b,r\in \NN$, $A^r_{a,b} =  \binom{2r}{2a+2}$, and  $B^r_{a,b} =\bigl(1-2^{-2r}\bigr)\binom{2r}{2b+1}$.
\end{theorem}
Zagier's proof of this theorem involves  an ingenious mixture of analytic techniques. 
The next step in the proof of theorem $\ref{thmHoffMZVLi}$ is to lift Zagier's theorem to the level of motivic multiple zeta values by checking its compatibility  with the
coaction $(\ref{mainformula})$ and using theorem \ref{thmprimitives}.  Since then, the proof of theorem \ref{thmZagierthm} was simplified by Li \cite{Li}, and 
Terasoma  \cite{TerasomaBZ} has verified that it can be deduced from associator equations. Since the associator equations are known to hold
between motivic multiple zeta values, it follows that, in principle, this part of the proof can now be deduced
directly by elementary methods (i.e., without using theorem \ref{thmprimitives}).

From the motivic version of theorem \ref{thmZagierthm}, one can compute  the action of the abelianization of  $U^{dR}$  on the  vector space built out of the elements  $\zetam(n_1,\ldots, n_r)$, with $n_i =2,3$, graded by the number of $3$'s. This action can be expressed by certain matrices constructed out of the combinatorial formula $(\ref{mainformula})$, whose entries are linear combinations of the coefficients $A^r_{a,b}$ and $B^r_{a,b}$ of theorem \ref{thmZagierthm}.
The key point is that these matrices have non-zero determinant $2$-adically, and are hence invertible. At its heart,  this uses the fact that the $B^r_{a,b}$ terms 
in theorem \ref{thmZagierthm}  dominate  with respect to the $2$-adic norm due to the factor $2^{-2r}$.

\section{Roots of unity} \label{sectRoots}

There are a handful of  exceptional cases when one knows how to generate certain categories of mixed 
Tate motives over  cyclotomic fields and write down their periods. These results are due to Deligne \cite{DeRoots}, inspired by numerical computations due to Broadhurst 
in 1997 relating to computations of  Feynman integrals.

Let $N\geq 2$ and let $\mu_N$ be the group of $N^{\mathrm{th}}$ roots of unity, and consider 
\begin{equation}\label{P1minusmuN}
\Pro^1 \backslash \{0,\mu_N, \infty\}
\end{equation}
Fix a primitive $N^{\mathrm{th}}$ root $\zeta_N$.  One can consider the corresponding motivic fundamental groupoid (with respect to suitable tangential base points)
and ask whether it generates the category $\MT(\Or_N[\textstyle{1 \over N}])$, where $\Or_N$ is the ring of integers in the field $\QQ(\zeta_N)$.
Goncharov has shown that for many primes $N$, and in particular, for $N=5$, this is false: already in weight two, there are motivic periods of this category which cannot
be expressed as motivic iterated integrals on  $\Pro^1 \backslash \{0,\mu_N, \infty\}$.

In certain exceptional cases, Deligne has proven a stronger statement:
\begin{theorem} For $N=2,3,4,6$ (resp. $N=8$)  the motivic fundamental group
$$\pi_1^{mot}( \Pro^1\backslash \{0,1,\infty\}, \tone_0, \zeta_N) \qquad \big(\hbox{resp. }  \pi_1^{mot}( \Pro^1\backslash \{0,\pm 1,\infty\}, \tone_0, \zeta_8) \big)$$
generates the categories $\MT(\Or_N[\textstyle{1 \over N}])$  for   $N =2,3,4, 8$,  and 
$ \MT(\Or_N)$   for  $N = 6$.   \end{theorem}

Iterated integrals on  $(\ref{P1minusmuN})$ can be expressed in terms of   cyclotomic multiple zeta values\footnote{The conventions in \cite{DeRoots} are opposite to the ones used here} which are defined for $(n_r, \varepsilon_r) \neq (1,1)$ by the sum
$$\zeta(n_1,\ldots, n_r ; \varepsilon_1, \ldots, \varepsilon_r) = \sum_{0< k_1 < k_2 <\ldots <k_r} {\varepsilon_1^{k_1} \ldots \varepsilon_r^{k_r}
\over k_1^{n_1} \ldots k_r^{n_r}} $$
where $\varepsilon_{1},\ldots, \varepsilon_{r}$ are roots of unity. The weight is defined as the sum of the indices $n_1+
\ldots +n_r$ and the depth is the increasing filtration defined by the integer $r$.
 It is customary to use the  notation  
$$\zeta(n_1,\ldots, n_{r-1},   n_r \zeta_N) = \zeta(n_1,\ldots, n_r ; \underbrace{1,\ldots , 1}_{r-1} , \zeta_N)\ . $$
One can define motivic versions 
relative to the canonical fiber functor $\omega$ (\cite{DG}, \S1.1) playing the role of what was previously the de Rham fiber functor (the two are related by   $\omega_{dR} = \omega\otimes \QQ(\zeta_N)$), and  the Betti realisation functor which corresponds to the   embedding $ \QQ(\zeta_N) \subset \CC$. Denote these motivic periods  by a superscript $\mm$.
Recall that  $\Lef^{\mm}$ is the motivic Lefschetz period of example \ref{exampleLef},   whose period is $2 \pi i$.
Let $X_{odd} = \{1,3,5,\ldots \}$ with the ordering $1>3>5 \ldots $.
Rephrased in the language of motivic periods, Deligne's results for $N=2,3,4$ yield:
\begin{enumerate}

\item($N=2$; algebra generators).  The following set of motivic periods:
$$\{ \Lef^{\mm} \} \cup \{ \zetam(n_1,\ldots,  n_{r-1}, - n_{r})\hbox{ where } (n_r,\ldots, n_1) \in \mathrm{Lyn}(X_{odd})\} $$ 
are algebraically independent over $\QQ$. The monomials in these quantities  form a  basis for the ring of  geometric motivic periods\footnote{recall that this is the subring of all motivic periods of  the category $\MT(\ZZ[{1 \over 2}])$ which is  generated by motives $M$ which have non-negative weights, i.e., $W_{-1}M=0$.}  of $\MT(\ZZ[{1 \over 2}])$.

\item($N=3, 4$;  linear basis).  The set of motivic periods
$$\zeta^{\mm}(n_1,\ldots, n_{r-1},   n_r \zeta_N) (\Lef^{\mm})^p \qquad \hbox{ where } n_i \geq 1, p\geq 0$$
are linearly independent over $\QQ$. They form a basis for the space of geometric motivic  periods of  $\MT(\Or_N  [ \textstyle{1 \over N}])$, for $N=3,4$ respectively.

\end{enumerate} 

By applying the period map, each case gives  a statement about cyclotomic multiple zeta values. In the case $N=2$, 
the underlying field is still $\QQ$, and 
it follows from $(i)$ that every multiple zeta value at $2^{\mathrm{nd}}$ roots of unity (sometimes called an Euler sum) is a polynomial with rational coefficients in
$$(2 \pi i)^2  \quad \hbox{ and } \quad \zeta(n_1,\ldots,n_{r-1},  -n_{r})  \qquad  \hbox{ where } \quad (n_1,\ldots, n_r) \in \mathrm{Lyn}(X_{odd})\ .$$
This decomposition respects the weight and depth, where the depth of $(2 \pi i)^n$ is $1$. Thus an Euler sum of weight $N$ and depth $r$
can be expressed as a polynomial in the above elements, of total weight $N$ and total depth $\leq r$.

\section{Depth} \label{sectDepth}
The results of the previous section for $N=2,3,4,6,8$ crucially use the fact that the depth filtration is dual to the lower central series of the corresponding
motivic Galois group. A fundamental difference with the case $N=1$ is that this fact is false 
for $\Pro^1 \backslash \{0,1,\infty\}$, due to a defect   closely related to modular forms.

Recall that   $\opi$  is a group for the Ihara action $\circ$. Let $\mathrm{Lie } (\opi)$ denote its  Lie algebra. Its bracket is denoted by $\{\  , \  \}$. 
Denote the images of the canonical generators \S\ref{sectCanGen}  by  $\sigma_{2n+1} \in \mathrm{Lie } (\opi)(\QQ)$, for $n\geq 1$.
They are elements of the free graded Lie algebra on two generators $x_0, x_1$, and we have, for example
$$\sigma_3 = [x_0,[x_0,x_1]] +  [x_1,[x_1,x_0]] $$   
The higher $\sigma_{2n+1}$ are of the form
$\sigma_{2n+1} = \mathrm{ad}(x_0)^{2n} (x_1)$  plus  terms of degree $\geq 2$ in $x_1$, 
but are not known explicitly except for small $n$.
By theorem \ref{thmfaithfulactionofGdr}, the $\sigma_{2n+1}$ freely generate a graded Lie subalgebra of $\mathrm{Lie } (\opi)(\QQ)$  which we denote by $\gm$. 
The depth filtration $\dd$ on $\gm$ is the decreasing filtration given by the degree in the letter $x_1$.
In 1993,  Ihara and Takao observed that 
\begin{equation} \label{IharaTakao}
\{\sigma_3, \sigma_9\} - 3  \{\sigma_5, \sigma_7\}  = {691 \over 144} \,   e_\Delta \end{equation}
where $e_\Delta$ is an element with integer coefficients of depth $\geq 4$ (degree $\geq 4$ in $x_1$), and the coefficient  $691$ on the right-hand side is  the numerator of the  Bernoulli number $B_{12}$. The element $e_\Delta$ is sparse:\footnote{`most' of its coefficients are zero, see \cite{BrDepth}, \S8 for a closed formula for this element.} indeed,  computations in the early days gave the impression that the right-hand side is zero, although we now know
that the $\sigma_{2n+1}$ generate a free Lie algebra.

Relations such as  $(\ref{IharaTakao})$ show that the structure of $\gm$ is related to arithmetic, but more importantly  show
that  the associated depth-graded  Lie algebra $\gr_{\dd}\,  \gm$ is not free, since the left-hand side of $(\ref{IharaTakao})$ vanishes in $\gr^2_{\dd}\,  \gm$. The depth filtration on $\gm$ corresponds, dually, to the depth filtration on motivic multiple zeta values, and $(\ref{IharaTakao})$ implies 
that motivic multiple zeta values  of depth $\leq 2$ are insufficient to span
the space of all (real geometric) motivic periods of $\MT(\ZZ)$ in weight 12 (one needs to include elements of depth $\geq 4$ such as
$\zetam(2,2,2,3,3)$ in a basis). 
By counting dimensions, this can be interpreted  as   a relation, viz:
\begin{equation}\label{exoticrel}
 28 \, \zetam (3,9)+  150 \, \zetam(5,7) + 168  \, \zetam(7,5) = {5197\over 691} \zetam(12)  \ Ê.
\end{equation}
The corresponding relation for multiple zeta values was found in \cite{GKZ} and generalised to an infinite family corresponding to cuspidal cohomology classes of $\mathrm{SL}_2(\ZZ)$.  In particular, the family of motivic multiple zeta values
$$\zetam(2n_1+1,\ldots, 2n_r+1) \zetam(2k)$$
cannot be a basis for $\Ho$, although it has the right dimensions in each weight $(\ref{dNdef})$.
The Hoffman basis $(\ref{HoffmotMZVs})$ gets around such pathologies, since, for example,  its elements  in weight 12  have depths between four and six.

In 1997, Broadhurst and Kreimer made exhaustive numerical computations on the depth filtration of multiple zeta values, which led them to the 
following conjecture, translated into the language of motivic multiple zeta values.

\begin{conjecture} \label{BKconj} (Motivic version of the Broadhurst-Kreimer conjecture) Let $\dd$ denote the increasing filtration on $\Ho$ induced by the depth.  Then
\begin{equation} \label{BKconj}
\sum_{N, d\geq 0} \dim_{\QQ}\,  (gr^{\dd}_d \Ho_{N})\,  s^d t^N = { 1  + \mathbb{E}(t) s \over 1- \mathbb{O}(t) s + \mathbb{S}(t) s^2 - \mathbb{S}(t) s^4}\ ,
\end{equation}
where $\mathbb{E}(t) =\textstyle {t^2 \over 1-t^2}$,   $ \mathbb{O}(t) =\textstyle{t^3 \over 1-t^2} $, and $ \mathbb{S}(t) =\textstyle{t^{12} \over (1-t^4)(1-t^6)} .$
\end{conjecture} 

Note that   equation $(\ref{BKconj})$ specializes to $(\ref{dNdef})$ on setting  $s$ equal to  $1$.
The   series  $\mathbb{E}(t)$ and  $\mathbb{O}(t)$ are  the generating series  for the dimensions of the spaces of even and odd single motivic zeta values.  The interpretation of $\mathbb{S}(t)$ as the generating series for cusp forms for $\mathrm{SL}_2(\ZZ)$ suggests a deeper connection with modular forms which is well understood in depth two.  By work of Zagier, and Goncharov,  formula $(\ref{BKconj})$ has been confirmed in depths $2$ and $3$ (i.e., modulo $s^4$).
 
 An interpretation for conjecture $(\ref{BKconj})$ in terms of the structure of $\gr_{\dd} \gm$, as well as a complete conjectural description 
  of  generators and relations of $\gr_{\dd} \gm$ in terms of modular forms for $\mathrm{SL}_2(\ZZ)$ was given in \cite{BrDepth}.
A deeper geometric understanding of this conjecture  would seem to require a 
framework which places multiple zeta values and modular forms on an equal footing, which is the topic of \S\ref{sectMMV}.

\section{Multiple modular values} \label{sectMMV}
In this final paragraph, I want  suggest applying the philosophy of \S\ref{sectTransperiods} to  iterated integrals on (orbifold) quotients of the upper half plane
$$\HH = \{ \tau \in \CC  : \mathrm{Im\, } (\tau) > 0 \}$$
by finite index subgroups $\Gamma \leq \mathrm{SL}_2(\ZZ)$. Iterated integrals of modular forms were first studied by Manin \cite{Ma1, Ma2}.
Here, I shall only consider the case $\Gamma = \mathrm{SL}_2(\ZZ)$.

\subsection{Eichler-Shimura integrals}
Denote the space of homogenous   polynomials of degree $n\geq 0$ with rational coefficients by 
$$V_n = \bigoplus_{i+j =n}   \QQ X^i Y^j $$
It admits a right action of $\Gamma$ via the  formula $(X,Y)|_{\gamma} = (aX+bY, cX+dY)$,
where $\gamma = \left( \begin{smallmatrix} a&b\\ c&d \end{smallmatrix} \right)$.
 Let $f(\tau)$ be a modular form of weight $k$ for $\Gamma$. Define
$$\underline{f}(\tau) = (2 \pi i)^{k-1} f(\tau) (X - \tau Y)^{k-2} d\tau \qquad \in \qquad \Gamma(\HH, \Omega^1_{\HH} \otimes V_{k-2}) $$
It  satisfies the  invariance property
 $ \underline{f}(\gamma(\tau))\big|_{\gamma} = \underline{f}(\tau)$ for all $\gamma \in \Gamma$.
 For $f$ a cusp form,   the classical Eichler-Shimura integral (see, e.g., \cite{KZ}) is
\begin{equation}\label{ESintegral}
 \int_0^{\infty} \underline{f}(\tau)  =  \sum_{n=1}^{k-1}  c_n L(f,n) X^{k-n-1} Y^{n-1} 
 \end{equation}
 where $c_n$ are certain explicit constants (rational multiples of a power of $\pi$)
 and $L(f,s)$ is the analytic continuation of the $L$-function $L(f,s) = \sum_{n\geq 1} {a_n \over n^s}$ of $f$,
 where $f(\tau) = \sum_{n\geq 1} a_n q^n$ and $q=e^{2 \pi i \tau}$.
 Manin showed that if $f$ is a Hecke eigenform, there  exist   $\omega^{+}_f, \omega^{-}_f \in \RR$
such that
 $$ \int_0^{\infty} \underline{f}(\tau)  = \omega^+_f P_f^+(X,Y) + \omega^-_f P_f^-(X,Y) $$
where   $P_f^{\pm}(X,Y) \in V_{k-2}\otimes \overline{\QQ}$ are  polynomials  with algebraic coefficients
which are invariant (resp. anti-invariant) with respect to $(X,Y) \mapsto (-X,Y)$.
 
 Recall that the Eisenstein series of weight $2k$, for $k\geq 2$, is defined by
\begin{equation}\nonumber
e_{2k} (q) = - {B_{2k} \over 4k} + \sum_{ n \geq 1} \sigma_{2k-1}(n) q^n \ , \qquad q=e^{2 \pi i \tau}
\end{equation}
where $B_{2k}$ is the $2k^{\mathrm{th}}$ Bernoulli number, and $\sigma$ denotes the divisor function. The corresponding integrals for Eisenstein series diverges. 
Zagier showed how to extend the definition of the Eichler-Shimura integrals to the case  $e_{2k}$, giving \cite{KZ}

\begin{equation}
  \quad { (2k-2)!  \over 2} \zeta(2k-1)  (Y^{2k-2} - X^{2k-2})   - {  (2\pi i)^{2k-1} \over  4k (2k-1)} \sum_{a+b=2k, a, b\geq 1} \binom{2k}{a} B_a B_b X^{a-1} Y^{b-1}  \label{Zagint}
 \end{equation}

 Manipulating this formula  leads to expressions for the odd Riemann zeta values
 in terms of Lambert series similar to  the following formula due to Ramanujan:
 $$\zeta(3) = {7 \over 180} \pi^3 - 2 \sum_{n \geq 1} {1 \over n^3 (e^{2n  \pi } -1)}\ .$$
It converges very rapidy.  One wants to think  of  $(\ref{Zagint})$ as pointing towards  a modular construction of $\zetam(2k-1)$.
 
\subsection{Regularisation} The theory of tangential base points (\cite{DeP1}, \S15) gives a general procedure for regularising iterated integrals on curves.
If one applies this to the orbifold $\Gamma \bq \HH $,  where $\Gamma = \mathrm{SL}_2(\ZZ)$,
one can show that it yields the  completely explicit formulae below,  which generalise  Zagier's formula for a single Eisenstein series. I shall only state the final answer.
Via the map
$$\tau \mapsto q= \exp(2 i \pi \tau) : \HH \longrightarrow   \{q \in \CC: 0< |q| < 1 \}=D^{\times}$$
a natural choice of tangential base point (denoted $\tone_{\infty}$) corresponds to the  tangent vector $1$ at $q=0$. 
Since in this case we have  explicit models  $\HH\subset \CC$ for a universal covering space of $\Gamma \bq \HH$, and 
$\CC$ for the universal covering of $D^{\times}$,  one can   compute all regularised  iterated integrals  by pulling them back to  $\CC$ as follows.

First, if $f= \sum_{n\geq 0 } f_n q^n$ is the Fourier expansion of $f$,  write
\begin{equation} \label{finfinity}
\underline{f}^{\infty}(\tau) = (2  \pi i )^{k-1} f_0 (X- \tau Y) ^{k-2}  d \tau \qquad \in \qquad \Gamma(\CC, \Omega^1_{\CC} \otimes V_{k-2})
\end{equation}
Define a linear  operator  $R$ on the tensor coalgebra on $ \Gamma(\CC, \Omega^1_{\CC} \otimes V)$ by
\begin{align} 
R [  \omega_1 | \ldots | \omega_n] &  = \sum_{i=0}^n (-1)^{n-i} [\omega_1 | \ldots | \omega_i] \sha [\omega_n^{\infty} | \ldots | \omega_{i+1}^{\infty}]  \nonumber \\
& = \sum_{i=1}^n (-1)^{n-i}\Big[ [\omega_1 | \ldots | \omega_{i-1}] \sha [\omega_n^{\infty} | \ldots | \omega_{i+1}^{\infty}] \Big| \omega_i- \omega_i^{\infty}\Big]    \ .\nonumber 
\end{align}
where $V= \bigoplus_k V_k$ and $\omega^{\infty}$ is the `residue at infinity' of $\omega$ defined by $(\ref{finfinity})$.    
The  regularised iterated integral  can be expressed as \emph{finite} integrals
$$\int_{\tau}^{\tone_{\infty}} [ \underline{\omega}_1 | \ldots | \underline{\omega}_n ] = \sum_{i=0}^n \int_{\tau}^{\infty} R  [ \underline{\omega}_1 | \ldots | \underline{\omega}_i] \int_{\tau}^{0}  [  \underline{\omega}^{\infty}_{i+1} | \ldots | \underline{\omega}_{n}^{\infty} ]$$
It takes values in $V_{k_1-2} \otimes \ldots \otimes V_{k_n-2}\otimes \CC$ if $\omega_1, \ldots, \omega_n$ are of weights $k_1,\ldots k_n$, 
and hence admits a right action of $\Gamma$. 
The integrals in the right factor on the right-hand side are  simply  polynomials in $\tau$ and can be computed explicitly.

\subsection{Cocycles} Choose a basis of Hecke normalised eigenforms $f_i$ indexed by  non-commuting symbols $A_i$, and  form the  generating series
$$I(\tau; \infty) = \sum_{i_k, n\geq 0} A_{i_1} \ldots A_{i_n} \int_{\tau}^{\tone_{\infty}} [ \underline{\omega}_{i_1} | \ldots | \underline{\omega}_{i_n} ] $$
For every $\gamma \in \Gamma$, there exists a formal power series $C_{\gamma}$ in the $A_i$ such that
\begin{equation}\label{Cdef}
I(\tau; \infty)  =  I(\gamma(\tau);\infty)|_{\gamma}\,  C_{\gamma}   
 \end{equation}
which does not depend on $\tau$.  It satisfies the cocycle relation
$$ C_{gh} = C_g\big|_h \, C_h \quad \hbox{ for all } g,h \in \Gamma\ .$$
The part of the cocycle $C$ which involves iterated integrals of cusp forms  was previously considered by Manin \cite{Ma1,Ma2}.
Since the group $\Gamma$ is generated by 
$$
S= 
\left(
\begin{array}{cc}
  0   & -1  \\
   1  &   0 
\end{array}
\right)\quad, \quad  T= \left(
\begin{array}{cc}
  1   &  1  \\
   0  &   1 
\end{array}
\right)\ , 
\ 
$$
the cocycle  $C$ is  determined by $C_S$ and $C_T$.  The series $C_T$ can be computed explicitly 
and its coefficients lie in  $\QQ[2\pi i]$. 
\begin{defn} Define the ring of multiple modular values with respect to the group $\Gamma = \mathrm{SL}_2(\ZZ)$ to be the subring of $\CC$  generated by the coefficients of $C_S$.
\end{defn}

The series $C_S$ is a kind of analogue of Drinfeld's associator $\mathcal{Z}$. Its terms of degree 1 in the $A_i$
are precisely the   Eichler-Shimura integrals $(\ref{ESintegral})$ and  $(\ref{Zagint})$.
Setting $\tau=i$ in $(\ref{Cdef})$ gives  integrals which converge extremely fast and are very well suited to  numerical computation.

\subsection{Galois action} One can mimic the Betti-de Rham aspects of the theory of the motivic
fundamental group of $\Pro^1 \backslash\{0,1,\infty\}$ as follows: 
\begin{enumerate}
\item The coefficients of $C_S$ can be interpreted as certain periods of the relative unipotent completion of $\Gamma$. This was defined  by Deligne as follows. Let $k$ be a field of characteristic $0$ and $S$ a reductive  algebraic group over $k$.
Suppose that $\Gamma$ is a discrete group equipped with  a Zariski dense homomorphism
$\rho: \Gamma \rightarrow S(k).$
The completion of $\Gamma$ relative to $\rho$ is an affine algebraic group scheme $ \mathcal{G}_{\Gamma}$, 
 which sits in an exact sequence
$$ 1 \longrightarrow \mathcal{U}_{\Gamma} \longrightarrow \mathcal{G}_{\Gamma} \longrightarrow S \longrightarrow 1$$
where $\mathcal{U}_{\Gamma}$ is pro-unipotent. 
There is a natural map $\Gamma \rightarrow \mathcal{G}_{\Gamma}(k)$ which is Zariski dense, and  whose projection onto $S(k)$ is the map $\rho$.

\item   In `geometric' situations, one expects the relative completion  to be the Betti realisation of something which is motivic.
Indeed, Hain has shown \cite{HaMHS}, \cite{HaGPS} that $\Or(\mathcal{G}_{\Gamma})$ carries a mixed Hodge structure in this case.
As a result, one can define Hodge-motivic periods  and try to carry out \S\ref{sectTransperiods}.

\item  The action of the unipotent radical of the Tannaka group of mixed Hodge structures  acts via the automorphism group of a  space of non-abelian cocyles 
of $\Gamma$ with coefficients in $\mathcal{U}_{\Gamma}$. It is  a certain semi-direct product of $\mathcal{U}_{\Gamma}$ with a group of non-commutative substitutions $\mathrm{Aut}(\mathcal{U}_{\Gamma})^S$.  An inertia condition  corresponds, in the case  $\Gamma = \mathrm{SL}_2(\ZZ)$, to the fact that $C_T$
is fixed, and  there are further constraints coming from the action of Hecke operators.  
  The explicit expression for $C_T$ yields precise information about the action of the Hodge-Galois group.
\end{enumerate}

The following key example illustrates how  multiple modular values for $\mathrm{SL}_2(\ZZ)$ resolve the depth-defect for
multiple zeta values as discussed in \S\ref{sectDepth}.

\begin{example}  \label{examplewt12} On $\Pro^1 \backslash \{0,1,\infty\}$ there are $2^{12}$ integrals of weight $12$, namely 
$$\int_{dch} \omega_{i_1} \ldots \omega_{i_{12}} \quad  \hbox{ where  }  \quad  \omega_{i_j} \in \{ {dt \over t}, {dt \over 1-t}\} \ . $$ 
However the space of multiple zeta values $\mathcal{Z}_{12}$ in weight $12$ has dimension at most $d_{12} = 12$, so there are a huge number of relations. Indeed, modulo products of multiple zeta values of lower weights, there are at most two elements of weight 12:
\begin{equation} \label{zetainweight12}
\zeta(3,3,2,2,2) \qquad \hbox{ and } \qquad  \zeta(3,2,3,2,2)
\end{equation}
by the corollary to theorem \ref{thmAlgInd}.  They are conjectured to be algebraically independent. Note that multiple zeta values of depths $\leq 2$ (or $\leq 3$ for that matter) will not suffice to span $\mathcal{Z}_{12}$ by equation $(\ref{exoticrel})$.

On the other hand, we can consider the coefficients of $C_S$ corresponding to regularised iterated integrals of Eisenstein series 
\begin{equation}\label{eisint}
\int_0^{\tone_{\infty}} \underline{e}_{2a}( X,Y)\underline{e}_{2b}( X,Y) \in \CC[X,Y]
\end{equation}
If we are interested in  periods modulo products, there are just two  relevant cases: $(2a,2b)  \in \{(4,10), (6,8)\}$. 
The description 3 above enables one to extract the relevant  numbers from the coefficients of these polynomials. One finds experimentally that  one 
obtains exactly  the  elements $(\ref{zetainweight12})$ modulo products, and that this is consistent with the coaction on the corresponding Hodge-motivic periods. Thus $\mathcal{Z}_{12}$ is  spanned by  exactly the right number of  multiple modular values (which are linear combinations of the coefficients of  $(\ref{eisint})$).
\end{example}

The example shows that in weight $12$, there are exactly two multiple modular values (modulo products)  which are multiple zeta values, and they conjecturally satisfy no relations. By contrast, multiple zeta values in weight 12 are hugely over-determined, and satisfy a vast number of relations.
Furthermore, the depth-defect described in \S5 can be directly related to the appearance of  special values of $L$-functions of cusp forms  amongst certain coefficients of (6.7).
\\

In conclusion, a rather optimistic hope is  that a  theory of motivic multiple modular values for congruence subgroups of $\mathrm{SL}_2(\ZZ)$  might provide a more natural construction of the periods of mixed Tate motives over cyclotomic fields (and much more) than the motivic fundamental groupoid of  the projective line minus Nth roots of unity, which suffers from the depth defect in the case N=1 (\S5), and from absent periods in non-exceptional cases such as N=5 (\S4).

\section{References}


\end{document}